\newif\ifarxiv
\newif\ifarxivcolor
\newif\ifpublication
\newif\ifpublicationcolor

\documentclass[a4paper, english, 12pt, reqno]{amsart}

\arxivtrue 
\arxivcolortrue 
\publicationtrue 
\publicationcolortrue 

\usepackage{amssymb}

\usepackage[english]{babel}
\usepackage[utf8]{inputenc}
\usepackage{tikz,tikzscale}
\tikzset{>=latex}
\usepackage{xcolor}
\usepackage{enumerate}
\usepackage{booktabs,multirow}

\newtheorem{theorem}{Theorem}[section]
\newtheorem{lemma}[theorem]{Lemma}

\theoremstyle{remark}

\numberwithin{equation}{section}

\newcommand{\elem}{\ensuremath{T}}
\newcommand{\mesh}{\ensuremath{\mathcal T}}

\newcommand{\faceSet}{\ensuremath{\mathcal F}}
\newcommand{\faceSetDir}{\ensuremath{\mathcal F^\textup D}}
\newcommand{\face}{\ensuremath{F}}
\newcommand{\skeletal}{\ensuremath{\Sigma}}
\newcommand{\skeletalSpace}{\ensuremath{\tilde M}}

\newcommand{\contElementSpace}{\ensuremath{V^\textup c}}
\newcommand{\contThreeElementSpace}[1]{\ensuremath{V^\textup c_{#1,p+3}}}
\newcommand{\linElementSpace}{\ensuremath{\overline V^\textup c}}
\newcommand{\discElementSpace}{\ensuremath{V}}
\newcommand{\polynomials}{\ensuremath{\mathcal P}}
\newcommand{\level}{\ensuremath{\ell}}
\newcommand{\iterMgOuter}{i}
\newcommand{\iterMgInner}{m}

\newcommand{\Div}{\nabla\!\cdot\!}

\newcommand{\extensionOp}{\ensuremath{\mathcal U^\textup c}}

\newcommand{\traceOp}{\ensuremath{\gamma_\level}}
\newcommand{\injectionOp}{\ensuremath{I}}
\newcommand{\projectionOp}{\ensuremath{P}}
\newcommand{\projectionOrthogonalOP}{\ensuremath{\Pi}}

\newcommand{\skeletalProj}{\ensuremath{\Pi^\partial}}
\newcommand{\contLinProj}{\ensuremath{\overline \Pi^\textup c}}
\newcommand{\discProj}{\ensuremath{\Pi^\textup d}}
\newcommand{\liftingOp}{\ensuremath{S}}

\renewcommand{\vec}[1]{\ensuremath{\boldsymbol{#1}}}
\newcommand{\Nu}{\ensuremath{\vec \nu}}
\newcommand{\dx}{\ensuremath{\, \textup d x}}
\newcommand{\ds}{\ensuremath{\, \textup d \sigma}}
\newcommand{\localU}{\ensuremath{\mathcal U}}
\newcommand{\localQ}{\ensuremath{\vec{\mathcal Q}}}

\newcommand{\llangle}{\ensuremath{\langle \! \langle}}
\newcommand{\rrangle}{\ensuremath{\rangle \! \rangle}}
\newcommand{\nnorm}{\ensuremath{\vert \! \vert \! \vert}}

\usepackage[colorlinks = true, linkcolor = blue, citecolor = blue, urlcolor = blue]{hyperref}
\usepackage{cleveref}

\begin{document}

\title[HMG for EDG]{Homogeneous multigrid for embedded discontinuous Galerkin methods} 

\author{Peipei Lu}
\address{Department of Mathematics Sciences, Soochow University, Suzhou, 215006, China}
\email{pplu@suda.edu.cn}
\thanks{P.~Lu has been supported by the Alexander von Humboldt Foundation.}

\author{Andreas Rupp}
\address{Interdisciplinary Center for Scientific Computing (IWR), Heidelberg University, Mathematikon, Im Neuenheimer Feld 205, 69120 Heidelberg, Germany}
\email{andreas.rupp@fau.de, andreas.rupp@uni-heidelberg.de}
\thanks{This work is supported by the Deutsche Forschungsgemeinschaft (DFG, German Research Foundation) under Germany's Excellence Strategy EXC 2181/1 - 390900948 (the Heidelberg STRUCTURES Excellence Cluster).}

\author{Guido Kanschat}
\address{Interdisciplinary Center for Scientific Computing (IWR) and Mathematics Center Heidelberg (MATCH), Heidelberg University, Mathematikon, Im Neuenheimer Feld 205, 69120 Heidelberg, Germany}
\email{kanschat@uni-heidelberg.de}

\subjclass[2010]{65F10, 65N30, 65N50}

\date{\today}


\begin{abstract}
 We introduce a homogeneous multigrid method in the sense that it uses the same embedded discontinuous Galerkin (EDG) discretization scheme for Poisson's equation on all levels. In particular, we use the injection operator developed in \cite{LuRK2020} for HDG and prove optimal convergence of the method under the assumption of elliptic regularity. Numerical experiments underline our analytical findings. 
 \\[1ex] \noindent \textsc{Keywords.}
 Multigird method, embedded discontinuous Galerkin, Poisson equation.
\end{abstract}
\maketitle
%
%
\section{Introduction}
As described in \cite{CockburnGSS2009}, the embedded discontinuous Galerkin (EDG) method can be obtained from the hybridizable discontinuous Galerkin (HDG) methods by replacing the space for the hybrid unknown by an overall continuous space. Thus, the stiffness matrix is significantly smaller, and its size and sparsity structure coincide with those of the stiffness matrix of the statically condensed continuous Galerkin method. Additionally, the condition number of the resulting EDG system is smaller than the one of the HDG system. However, \cite{CockburnGSS2009} underlines that the computational advantage has to be balanced against the fact that the approximate solutions of the primary and flux unknowns both lose a full order of convergence.

EDG schemes and their variants have gained some popularity over the last decade. They have, for example, been successfully applied to advection--diffusion \cite{Fu2017}, Stokes \cite{Rhebergen2020}, Euler and Navier--Stokes equations \cite{PeraireNC2011, NguyenPC2015}, distributed optimal control for elliptic problems \cite{Zhang2018}, Dirichlet boundary control for advection--diffusion \cite{Chen2019}, and compared to stabilized, residual-based finite elements \cite{Kamenetskiy2016}. However, to the best of our knowledge, no multigrid method is available for EDG schemes. Thus, we propose the first (homogeneous) multigrid method for the embedded discontinuous Galerkin method.

Homogeneous multigrid methods use the same discretization scheme on all levels. Such methods are important, since they have the same mathematical properties on all levels. They are also advantageous from a computational point of view, since their data structures and execution patterns are more regular. 

Our considerations are based on the analysis techniques for multigrid methods applied to HDG discretizations. The first of these methods has been introduced in \cite{CockburnDGT2013,TanPhD}, while similar results have been obtained for hybrid Raviart--Thomas (RT) schemes in \cite{GopalakrishnanTan09}. However, all these schemes fall back to linear finite elements and therefore cannot be called ``homogeneous''. The first homogeneous multigrid method for hybrid discontinuous Galerkin schemes has finally been introduced in \cite{LuRK2020}.

Thus, the structure of the analysis conducted in this manuscript is similar to the one in \cite{LuRK2020} and uses the same notation and some results from \cite{LuRK2020}, but the proof technique demonstrated in the following is significantly different.

The remainder of this paper is structured as follows: In Section \ref{SEC:basics}, we briefly review the EDG method for the considered elliptic PDE. Furthermore, an overview over the used function spaces, scalar products, and operators is given. Section \ref{SEC:multigrid} is devoted to a brief explanation of the multigrid and states the assumptions for its main convergence result. Sections \ref{SEC:proof_a1} and \ref{SEC:proof_a2_a3} verify the assumptions of the main convergence result, while Section \ref{SEC:numerics} underlines its validity by numerical experiments. Short conclusions wrap up the paper. 
%
\section{Model equation and discretization}\label{SEC:basics}
%
We consider the Dirichlet boundary value problem for Poisson's equation
\begin{equation}
 -\Delta u = f \quad \text{ in } \Omega, \qquad u = 0 \quad \text{ on } \partial \Omega
\end{equation}
defined on a polygonally bounded domain $\Omega \subset \mathbb R^d$. The flux vector is  $\vec q = - \nabla u$.  In the analysis, we will assume elliptic regularity, namely $u \in H^2(\Omega)$ if $f \in L^2(\Omega)$, such that there is a constant $c>0$ for which
\begin{gather}
  |u|_{H^2(\Omega)}
  \le c \| f \|_{L^2(\Omega)}
\end{gather}
holds. The domain $\Omega$ is discretized by a hierarchical sequence of
triangulations $\mesh_\level$ for $\level = 1,\dots, L$. We assume
that each simplicial mesh is topologically regular such that each facet of a cell
is either a facet of another cell or on the boundary. The sequence is
shape regular in the usual sense. The sequence is constructed
recursively from a coarse mesh $\mesh_0$ by refinement, such that each
cell of the mesh $\mesh_{\level-1}$ on level $\ell-1$ is the union of
cells of $\mesh_\level$. The meshes are assumed quasi-uniform such that the typical diameter of a cell of the mesh on level $\ell$ is $h_\ell$. Finally, we assume that
refinement from one level to the next is bounded in the sense that there
is a constant $c_\text{ref} > 0$ with
\begin{gather}
  h_\level \ge c_\text{ref} h_{\level-1}.
\end{gather}

By $\faceSet_\level$ we denote the set of faces of $\mesh_\level$.
The subset of faces on the boundary is
\begin{gather}
  \faceSetDir_\level := \{\face \in \faceSet_\level : \face \subset \partial \Omega \}.
\end{gather}
Moreover, we define
$\faceSet^\elem_\level := \{ \face \in \faceSet_\level : \face \subset
\partial \elem \}$ as the set of faces of a cell
$\elem\in\mesh_\level$.  We identify the set $\faceSet_\level$ as a
set of faces with the union of these faces as a subset of
$\overline\Omega$, such that the notion of function spaces
$C^k(\faceSet_\level)$ and $L^2(\faceSet_\level)$ are meaningful. The latter is equipped with the inner product
\begin{gather}
  \llangle \lambda, \mu \rrangle_\level = \sum_{\elem \in \mesh_\level} \int_{\partial \elem} \lambda\mu\ds,
\end{gather}
and its induced norm
$\nnorm \mu \nnorm^2_\level = \llangle \mu, \mu
\rrangle_\level$.
  Note that interior faces appear twice in this definition such that expressions like $\llangle u, \mu \rrangle_\level$ with possibly discontinuous $u|_{\elem} \in H^1(\elem)$ for all $\elem \in \mesh_\level$ and $\mu \in L^2(\faceSet_\level)$ are defined without further ado.
Additionally, we define an inner product
commensurate with the $L^2$-inner product in the bulk domain, namely
\begin{gather}
  \langle \lambda, \mu \rangle_\level
  = \sum_{\elem \in \mesh_\level} \frac{|\elem|}{|\partial \elem|}
  \int_{\partial \elem} \lambda \mu \ds \cong \sum_{\face \in \faceSet_\level} h_\face
  \int_{\face} \lambda \mu \ds.
\end{gather}
Its induced norm is $ \| \mu \|^2_\level = \langle \mu, \mu \rangle_\level$.

Let $p\ge 1$ and $\polynomials_p$ be the space of multivariate
polynomials of degree up to $p$.  EDG method can be obtained from
corresponding HDG methods by replacing the HDG skeletal space by the
EDG skeletal space
\begin{equation}
\skeletalSpace_\level := \left\{ \lambda \in C^0 (\faceSet_\level) \;\middle|\;
    \begin{array}{r@{\,}c@{\,}ll}
  \lambda_{|\face} &\in& \polynomials_p & \forall \face \in \faceSet_\level, \, \face \not\subset \partial \Omega \\
  \lambda_{|\face} &=& 0 & \forall \face \in \faceSet_\level, \, \face \subset \partial \Omega   
    \end{array}
  \right\},
\end{equation}

The EDG method involves a local solver on each mesh cell
$\elem \in \mesh_\level$ which can be understood as an approximate
Dirichlet to Neumann map on each mesh cell. It is written in mixed
form, producing cellwise approximate diffusion solutions
$u_\elem \in V_\elem$ and $\vec q_\elem\in \vec W_\elem$,
respectively, by solving for given boundary values $\lambda$
\begin{subequations}\label{EQ:hdg_scheme}
\begin{align}
  \int_\elem \vec q_\elem \cdot \vec p_\elem \dx - \int_\elem u_\elem \Div \vec p_\elem \dx
  & = - \int_{\partial \elem} \lambda \vec p_\elem \cdot \Nu \ds
    \label{EQ:hdg_primary}
  \\
  - \int_\elem \vec q_\elem \cdot \nabla v_\elem \dx  + \int_{\partial \elem} ( \vec q_\elem \cdot \Nu + \tau_\level u_\elem ) v_\elem \ds
  & = \tau_\level \int_{\partial \elem} \lambda v_\elem \ds \label{EQ:hdg_flux}
\end{align}
\end{subequations}
for all $v_\elem \in V_\elem$, and all $\vec p_\elem \in \vec W_\elem$. Here,
$\Nu$ is the outward unit normal with respect to $\elem$ and $\tau_\level > 0$
is the penalty coefficient of the method.

We choose $V_\elem = \polynomials_p$. Then, choosing
$\vec W_\elem = \polynomials_p^d$ yields an analogue of the so called
hybridizable local discontinuous Galerkin (LDG-H) scheme, i.e., the
embedded local discontinuous Galerkin (LDG-E) scheme.

Our current analysis is in fact limited to this case and other choices
require a modification of Lemma \ref{LEM:lem_35} and Lemma \ref{LEM:thm_31}.

While the local solvers are implemented
cell by cell, it is helpful for the analysis to combine them by
concatenation. To this end, we introduce the spaces
\begin{gather}
  \label{EQ:dg_spaces}
  \begin{aligned}
    \discElementSpace_\level
    &:=\bigl\{ v \in L^2(\Omega)
    & \big|\;v_{|\elem} &\in V_\elem,
    &\forall \elem &\in \mesh_\level \bigr\},\\
    \vec W_\level
    &:=\bigl\{ \vec q \in L^2(\Omega;\mathbb R^d)
    & \big|\;\vec q_{|\elem} &\in \vec W_\elem,
    &\forall \elem &\in \mesh_\level \bigr\}.    
  \end{aligned}
\end{gather}
 Hence, the local solvers define a mapping
\begin{gather}
  \begin{split}
    \skeletalSpace_\level & \to \discElementSpace_\level \times \vec W_\level\\
   \lambda &\mapsto (\localU_\level \lambda, \localQ_\level \lambda),
 \end{split}
\end{gather}
where for each cell $\elem\in \mesh_\level$ holds
$\localU_\level \lambda = u_\elem$ and
$ \localQ_\level \lambda = \vec q_\elem$. In the same way, we define
operators $\localU_\level f$ and $ \localQ_\level f$ for
$f\in L^2(\Omega)$, where now the local solutions are defined by
the system
\begin{subequations}\label{EQ:hdg_f}
  \begin{align}
    \int_\elem \vec q_\elem \cdot \vec p_\elem \dx - \int_\elem u_\elem \Div \vec p_\elem \dx
    & = 0
      \label{EQ:hdg_f_primary}
    \\
    - \int_\elem \vec q_\elem \cdot \nabla v_\elem \dx  + \int_{\partial \elem} ( \vec q_\elem \cdot \Nu + \tau_\level u_\elem ) v_\elem \ds
    & =  \int_{\elem} f v_\elem \dx.
      \label{EQ:hdg_f_flux}
\end{align}
\end{subequations}

Once $\lambda$ has been computed, the EDG approximation to the solution
of the Poisson problem and its gradient on mesh $\mesh_\level$ will be
computed as
\begin{gather}
  \begin{split}
    u_\level &= \localU_\level \lambda + \localU_\level f\\
    \vec q_\level &= \localQ_\level \lambda + \localQ_\level f
  \end{split}
\end{gather}

The global coupling condition is derived through a discontinuous
Galerkin version of mass balance and reads: Find
$\lambda \in \skeletalSpace_\level$, such that for all
$ \mu \in \skeletalSpace_\level$
\begin{equation}
  \sum_{\elem \in \mesh_\level}
  \sum_{\face \in \faceSet^\elem_\level \setminus \faceSetDir_\level}
   \int_\face \left( \vec q_\level \cdot \Nu
    + \tau_\level (u_\level - \lambda)\right) \mu \ds = 0.\label{EQ:hdg_global}
\end{equation}

EDG can be formulated in a condensed version as finding $\lambda \in \skeletalSpace_\level$ such that
\begin{subequations}\label{EQ:hdg_condensed}
\begin{equation}\label{EQ:hdg_condensedP}
 a_\level(\lambda, \mu) = b (\mu) \qquad \forall \mu \in \skeletalSpace_\level
\end{equation}
with the bilinear form $a_\level(\cdot,\cdot)$ and linear form $b_\level(\cdot)$ defined by
\begin{align}
  a_\level(\lambda, \mu) = & (\localQ_\level \lambda, \localQ_\level)_{\Omega}
                             + \llangle \tau_\level (\localU_\level \lambda - \lambda),
                             (\localU_\level \mu - \mu)\rrangle_\level
                             \label{EQ:bilinear_condensed},\\
  b_\level(\mu) = & ( \localU_\level\mu, f )_\Omega,
\end{align}
\end{subequations}
where $(\cdot,\cdot)_\Omega$ is the standard inner product
in $L^2(\Omega)$ and $L^2(\Omega;\mathbb R^d)$, respectively. Note that this bilinear form is defined in the same way as the one for the HDG method. And since $\skeletalSpace_\level$ is a subspace of the corresponding space of the HDG method, it is symmetric and positive definite~\cite{CockburnGL2009}. Hence, it is a scalar product, and induces a norm denoted by $\| \cdot \|_{a_\level}$.

We  associate an operator
$A_\ell\colon \skeletalSpace_\level \to \skeletalSpace_\level$ with
the bilinear form $a_\level(\cdot,\cdot)$ by the relation
\begin{equation}\label{EQ:def_A}
 \langle A_\level \lambda, \mu \rangle_\level = a_\level(\lambda, \mu) \qquad \forall \mu \in \skeletalSpace_\level.
\end{equation}
%
  
Additionally, we introduce
\begin{align}
 \linElementSpace_\level~:=~ & \{ u \in H^1_0(\Omega) \; : \; u|_\elem \in \polynomials_1(\elem) \;\; \forall \elem \in \mesh_\level \},
\end{align}
and the $L^2$ projections
\begin{align}
 \skeletalProj_\level \colon & H^2(\Omega) \cap H^1_0(\Omega) \to \skeletalSpace_\level, && \llangle \skeletalProj_\level u , \mu \rrangle_\level = \llangle \gamma_\level u, \mu \rrangle_\level & \forall \mu \in \skeletalSpace_\level,\\
 \contLinProj_\level \colon & H^1_0(\Omega) \to \linElementSpace_\level, && ( \contLinProj_\level u, w )_0 = (u, w)_0 & \forall w \in \linElementSpace_\level,\\
 \discProj_\level \colon & H^1(\Omega) \to \discElementSpace_\level, && ( \discProj_\level u, w )_0 = (u, w)_0 & \forall w \in \discElementSpace_\level
\end{align}
with trace operator $\gamma_\level$ to be used in our analysis. Obviously, we have
\begin{align}
 \| u - \contLinProj_\level u \|_\level + \| u - \skeletalProj_\level u \|_\level ~\lesssim~ & h_\level^2 |u|_2, && \text{(trace approx.)}\\
 \| u - \discProj_\level u \|_0 + \| u - \contLinProj_\level u \|_0 ~\lesssim~ & h_\level |u|_1, && \text{($L^2$ approx.)}\\
 | \contLinProj_\level u |_1 ~\lesssim~ & |u|_1. && \text{($H^1$ stab.)}
\end{align}

Here and in the following, $\lesssim$ has the meaning of smaller than or equal to up to a constant only dependent on the regularity constant of the mesh family and $c_\text{ref}$.
%
%
\section{Multigrid method and main convergence result}\label{SEC:multigrid}
%
We consider a standard (symmetric) V-cycle multigrid method for \eqref{EQ:hdg_condensed} (cf. \cite{BramblePX1991}). Since we deal with noninherited forms, we use~\cite{DuanGTZ2007} as a base for our convergence analysis. Section \ref{SEC:multigrid_algortith} recites the multigrid method and Section \ref{SEC:main_convergence_result} states the main convergence result. First, we recall an estimate for eigenvalues and condition numbers of the matrices $A_\level$.
\begin{lemma}\label{LEM:eigenvalue_bound}
 Suppose that $\mesh_\level$ is quasiuniform. Then, there are positive constants $C_1$ and $C_2$ independent of $\level$ such that
 \begin{equation}
  C_1 \| \lambda \|^2_\level \le a_\level ( \lambda, \lambda ) \le \beta_\level C_2 h_\level^{-2} \| \lambda \|^2_\level, \qquad \forall \lambda \in \skeletalSpace_\level,
 \end{equation}
 where $\beta_\level := 1 + (\tau_\level h_\level)^2$.
\end{lemma}
\begin{proof}
 This is a Corollary of \cite[Theo.~3.2]{CockburnDGT2013} exploiting the fact that the skeletal space of EDG is a subset of the skeletal space for HDG.
\end{proof}
This implies that for the stiffness matrix, we can bound the condition number $\kappa_\level$ by
\begin{equation}
 \kappa_\level \lesssim \beta_\level h_\level^{-2}
\end{equation}
which implies that for all choices of $\tau_\level$ satisfying $\tau_\level \lesssim h_\level^{-1}$ the condition number grows at most like $h_\level^{-2}$. 
%
\subsection{The injection operator $\injectionOp_\level$}
\label{SEC:injection}
%
The difficulty of devising an ``injection operator''
$I_\level: \skeletalSpace_{\level-1} \to \skeletalSpace_\level$ originates from the fact that the finer mesh
has edges which are not refinements of the edges of the coarse
mesh.
In order to assign reasonable values to these edges, we construct the injection operator similar to the HDG injection operator of \cite{LuRK2020} in three steps. First, we introduce the continuous finite element space
\begin{gather}
  \contElementSpace_\level := \bigl\{ u \in H^1_0(\Omega)
  \;\big|\;
  u_{|\elem} \in \polynomials_p(\elem) \quad \forall \elem \in \mesh_\level\bigr\},
  \label{EQ:cg_space}
\end{gather}
and define the shape function basis on each mesh cell $T$ by a Lagrange interpolation condition with respect to
support points $\vec x$. Afterwards, the continuous extension operator
\begin{equation}
  \extensionOp_\level\colon \skeletalSpace_{\level} \to \contElementSpace_{\level},
\end{equation}
can be defined using those interpolation conditions
\begin{gather}
  \label{EQ:define_uc}
  [\extensionOp_\level\lambda](\vec x) =
  \begin{cases}
    \lambda(\vec x) & \text{if $\vec x$ is located on a face}, \\
    [\localU_\level \lambda] (x) & \text{if $\vec x$ is in the interior of a cell}.
  \end{cases}
\end{gather}
Note that due to the continuity of the EDG method in vertices (and edges in three dimensions), no special handling of degrees of freedom there is needed and they are covered by the first line of the definition of $\extensionOp_\level\lambda$.

Since $\contElementSpace_{\level-1} \subset \contElementSpace_{\level}$, there is a natural embedding
\begin{equation}
  \begin{aligned}
    I_\level^c\colon \contElementSpace_{\level-1} &\to \contElementSpace_{\level} \\ u &\mapsto u.
  \end{aligned}
\end{equation}
On $\contElementSpace_{\level}$ the trace on edges is well defined, such that we can write
\begin{equation}
  \begin{aligned}
    \traceOp\colon \contElementSpace_{\level} &\to \skeletalSpace_{\level} \\ u &\mapsto \traceOp u.
  \end{aligned}
\end{equation}
Using these three operators, we define the injection operator $I_\level$ as their
concatenation, namely
\begin{equation}
  \begin{aligned}
    I_\level\colon \skeletalSpace_{\level-1} &\to \skeletalSpace_\level \\
    \lambda &\mapsto \traceOp I_\level^c \extensionOp_{\level-1} \lambda.
  \end{aligned}
\end{equation}


Since the EDG approximation space $\skeletalSpace_\level$ is a subspace of the HDG approximation space, the following Lemma is straightforward from \cite[Lem.~2.1]{LuRK2020}.
\begin{lemma}[Boundedness]
  \label{LEM:injection_bounded}
  When $\tau_\level = \tfrac{c}{h_\level}$, the injection operator $\injectionOp_\level$ is bounded in the sense that
  \begin{equation}
   a_\level(\injectionOp_\level \lambda , \injectionOp_\level \lambda) \lesssim a_{\level-1}(\lambda, \lambda) \qquad \forall\lambda \in \skeletalSpace_{\level-1}.
  \end{equation}
\end{lemma}
After the injection operator
  $\injectionOp_\level$ has been defined, we introduce two operators from
  $\skeletalSpace_{\level}$ to $\skeletalSpace_{\level-1}$, which
  replace the $L^2$-projection and the Ritz projection of conforming
  methods, respectively. They are $\projectionOrthogonalOP_{\level-1}$
  and $\projectionOp_{\level-1}$ defined by the conditions
\begin{xalignat}3
   \projectionOrthogonalOP_{\level-1}&\colon \skeletalSpace_\level \to \skeletalSpace_{\level-1},
   &\langle \projectionOrthogonalOP_{\level-1} \lambda, \mu \rangle_{\level-1}
   &= \langle \lambda, \injectionOp_\level \mu\rangle_\level
   && \forall \mu \in \skeletalSpace_{\level-1}.
   \label{EQ:L2_projection_definition}
   \\
   \projectionOp_{\level-1}&\colon \skeletalSpace_\level \to \skeletalSpace_{\level-1},
   &a_{\level-1}(\projectionOp_{\level-1} \lambda, \mu)
   &= a_\level(\lambda, \injectionOp_\level \mu)
   && \forall \mu \in \skeletalSpace_{\level-1},
   \label{EQ:projection_definition}
 \end{xalignat}
 The operator $\projectionOrthogonalOP_{\level-1}$ is used in the
 implementation, while $\projectionOp_{\level-1}$ is key to the
 analysis.

\subsection{Multigrid algorithm}\label{SEC:multigrid_algortith}
%
Assume that we have an injection operator $I_\level\colon\skeletalSpace_{\level-1} \to \skeletalSpace_\level$ for grid transfer. Actually, this has been defined in \cref{SEC:injection}. Assume further a smoother denoted by
 \begin{gather}
   R_\level: \skeletalSpace_\level \to \skeletalSpace_\level.
 \end{gather}
  In this manuscript we consider point smoothers in terms of Jacobi or Gauss-Seidel
  iterations, respectively. Denote by $R_\level^\dagger$ the adjoint operator of
  $R_\level$ with respect to
  $\langle \cdot, \cdot \rangle_\level$ and define $R_\level^\iterMgOuter$ by
 \begin{equation}
  R_\level^\iterMgOuter = \begin{cases} R_\level & \text{ if } \iterMgOuter \text{ is odd,} \\ R_\level^\dagger & \text{ if } \iterMgOuter \text{ is even.} \end{cases}
\end{equation}
Let $\iterMgInner \ge 1$ denote the number of smoothing steps. We recursively define the multigrid operator of the refinement level $\level$
\begin{equation}
 B_\level \; : \quad \skeletalSpace_\level \to \skeletalSpace_\level.
\end{equation}
First, $B_0 = A^{-1}_0$. For $\level > 0$ and  for $\mu\in\skeletalSpace_\level$ define 
$B_\level \mu$ as follows: let $x^0 = 0 \in \skeletalSpace_\level$.
\begin{enumerate}
 \item Define $x^\iterMgOuter \in \skeletalSpace_\level$ for $\iterMgOuter = 1, \ldots, \iterMgInner$ by
 \begin{equation}
  x^\iterMgOuter = x^{\iterMgOuter-1} + R_\level^{\iterMgOuter} ( \mu - A_\level x^{\iterMgOuter-1} ).
 \end{equation}
 \item Set $y^0 = x^\iterMgInner + \injectionOp_\level q$, where $q \in \skeletalSpace_{\level-1}$ is defined as
 \begin{equation}
  q = B_{\level-1} \projectionOrthogonalOP_{\level-1} ( \mu - A_\level x^\iterMgInner).
 \end{equation}
 \item Define $y^\iterMgOuter \in \skeletalSpace_\level$ for $\iterMgOuter = 1, \ldots, \iterMgInner$ as
 \begin{equation}
  y^\iterMgOuter = y^{\iterMgOuter - 1} + R^{\iterMgOuter+\iterMgInner}_\level ( \mu - A_\level y^{\iterMgOuter-1} ).
\end{equation}
\item  Let $B_\level \mu = y^{\iterMgInner}$.
\end{enumerate}
%
\subsection{Main convergence result}\label{SEC:main_convergence_result}
%
The analysis of the multigrid method is based on the framework
introduced in~\cite{DuanGTZ2007}. There, convergence is traced back to
three assumptions. Let $\underline \lambda^A_\level$ be the largest eigenvalue of $A_\level$, and
\begin{gather}
 K_\level := \bigl(1 - (1 - R_\level A_\level) (1 - R^\dagger_\level A_\level)\bigr) A^{-1}_\level.
\end{gather}
Then, there exists constants $C_1, C_2, C_3 > 0$ independent of the mesh level $\level$, such that there holds
\begin{itemize}
\item Regularity approximation assumption:
  \begin{equation}\label{EQ:precond1}
    | a_\level(\lambda - \injectionOp_\level \projectionOp_{\level-1} \lambda, \lambda) |
    \le C_1 \frac{\| A_\level \lambda \|^2_\level}{\underline \lambda^A_\level} \qquad \forall \lambda \in \skeletalSpace_\level. \tag{A1}
  \end{equation}
\item Stability of the ``Ritz quasi-projection'' $\projectionOp_{\level-1}$ and injection $\injectionOp_\level:$
 \begin{equation}\label{EQ:precond2}
  \| \lambda - \injectionOp_\level \projectionOp_{\level-1} \lambda\|_{a_\level} \le C_2 \| \lambda \|_{a_\level} \qquad \forall \lambda \in \skeletalSpace_\level. \tag{A2}
\end{equation}
\item Smoothing hypothesis:
 \begin{equation}\label{EQ:precond3}
  \frac{\| \lambda \|^2_\level}{\underline \lambda^A_\level} \le C_3 \langle K_\level \lambda, \lambda \rangle_\level. \tag{A3}
 \end{equation}
\end{itemize}
Theorem~3.1 in~\cite{DuanGTZ2007} reads
\begin{theorem}\label{TH:main_theorem}
 Assume that \eqref{EQ:precond1}, \eqref{EQ:precond2}, and \eqref{EQ:precond3} hold. Then for all $\level \ge 0$,
 \begin{equation}
  | a_\level ( \lambda - B_\level A_\level \lambda, \lambda ) | \le \delta a_\level(\lambda, \lambda),
 \end{equation}
 where
 \begin{equation}
  \delta = \frac{C_1 C_3}{\iterMgInner - C_1 C_3} \qquad \text{with} \qquad \iterMgInner > 2 C_1 C_3.
 \end{equation}
\end{theorem}

Thus, in order to prove uniform convergence of the multigrid method, we will now set out to verify these assumptions.
%
\section{Proof of \eqref{EQ:precond1}}\label{SEC:proof_a1}
%
To show \eqref{EQ:precond1}, we follow the lines of \cite[Sect.~4]{DuanGTZ2007} and verify the assumption of 
\begin{theorem}[Sufficient conditions for \eqref{EQ:precond1}]
 If $\tau_\level h_\level \lesssim 1$ and
 \begin{equation}\label{EQ:precond14}
  \| \lambda - \injectionOp_\level \projectionOp_{\level -1} \lambda \|_\level \lesssim h^2_\level \| A_\level \lambda \|_\level\tag{B1}
 \end{equation}
 holds for all $\lambda \in \skeletalSpace_\level$ and for all $\level$, \eqref{EQ:precond1} is satisfied.
\end{theorem}
\begin{proof}
 \begin{align}
  | a_\level (\lambda - \injectionOp_\level \projectionOp_{\level - 1} \lambda, \lambda) | \overset{\eqref{EQ:def_A}}{~=~} & | \langle \lambda - \injectionOp_\level \projectionOp_{\level - 1} \lambda, A_\level \lambda \rangle_\level |
  \\ ~\le~ & \| \lambda - \injectionOp_\level \projectionOp_{\level - 1} \lambda \|_\level \| A_\level \lambda \|_\level \overset{\eqref{EQ:precond14}} \lesssim h_\level^2 \| A_\level \lambda \|_\level^2 \notag \\
  \overset{\text{Lem.~\ref{LEM:eigenvalue_bound}}}{~\lesssim~} &\frac{\| A_\level \lambda \|_\level^2}{\underline \lambda^A_\level}.\notag
 \end{align}
\end{proof}
In order to prove \eqref{EQ:precond14}, we construct some auxiliary quantities. Let
\begin{gather}
     \contThreeElementSpace{\level}
     = \bigl\{ u \in H^1_0(\Omega) \;\big\vert\;
     u_{|\elem} \in \polynomials_{p+3}(\elem)
     \; \forall \elem \in \mesh_\level \bigr\}.
\end{gather}
For all $\lambda \in \skeletalSpace_\level$ define $\liftingOp_\level \lambda \in \contThreeElementSpace{\level}$ satisfying
\begin{subequations}
\begin{align}
 ( \liftingOp_\level \lambda, v)_\elem ~=~ & (\localU_\level \lambda, v )_\elem && \forall v \in \polynomials_p(\elem),\\
 \langle \liftingOp_\level \lambda, \eta \rangle_\face ~=~ & \langle \lambda, \eta \rangle_\face && \forall \eta \in \polynomials_{p+1} (\face), \face \subset \partial \elem,\\
 \liftingOp_\level \lambda (\vec a) ~=~ & \lambda(\vec a) && \vec a \text{ is a vertex of } \elem.
\end{align}
\end{subequations}
Actually, this is also the definition of the $H^1$-conforming finite element in \cite[Lem.~A.3]{GiraultR1986}.
By the standard scaling argument and using Lemma \ref{LEM:thm_31} when $\tau_\level h_\level \lesssim 1$, we have
\begin{equation}\label{EQ:norm_equiv}
 \| \liftingOp_\level \lambda \|_0 \cong \| \lambda \|_\level \qquad \forall \lambda \in \skeletalSpace_\level.
\end{equation}
That is, $(\liftingOp_\level \cdot , \liftingOp_\level \cdot)_0$ is an inner product on $\skeletalSpace_\level$. Thus, for all $\lambda \in \skeletalSpace_\level$, there is a $\phi_\lambda \in \skeletalSpace_\level$ such that
\begin{equation}\label{EQ:SCP_def}
 (\liftingOp_\level \phi_\lambda, \liftingOp_\level \mu)_0 = a_\level( \lambda, \mu ) = \langle A_\level \lambda, \mu \rangle_\level \qquad \forall \mu \in \skeletalSpace_\level.
\end{equation}
We denote $f_\lambda = \liftingOp_\level \phi_\lambda$, define $\tilde u$ as solution of
\begin{equation}\label{EQ:Poisson}
 - \Delta \tilde u = f_\lambda \quad \text{ in } \Omega, \qquad \tilde u = 0 \quad \text{ on } \partial \Omega
\end{equation}
and let $\tilde \lambda_\level \in \skeletalSpace_\level$ be the EDG approximation of \eqref{EQ:Poisson}, i.e.,
\begin{equation}\label{EQ:operator_Poisson}
 a_\level (\tilde \lambda_\level, \mu) = (f_\lambda, \localU_\level \mu)_0 \qquad \forall \mu \in \skeletalSpace_\level.
\end{equation}
From \eqref{EQ:norm_equiv} and \eqref{EQ:SCP_def}, we have for $\mu = \phi_\lambda$ that
\begin{equation}
 \| f_\lambda \|^2_0 = \| \liftingOp_\level \phi_\lambda \|^2_0 = \langle A_\level \lambda, \phi_\lambda \rangle_\level \le \| A_\level \lambda \|_\level \| \phi_\lambda \|_\level \lesssim \| A_\level \lambda \|_\level \| \liftingOp_\level \phi_\lambda \|_0
\end{equation}
which implies that
\begin{equation}\label{EQ:f_lambda_bound}
 \| f_\lambda \|_0 \lesssim \| A_\level \lambda \|_\level.
\end{equation}

\begin{lemma}\label{LEM:lifting_identity}
 If $w \in \linElementSpace_\level$, then
 \begin{equation}
  \liftingOp_\level \gamma_\level w = \localU_\level \gamma_\level w = w.
 \end{equation}
\end{lemma}

\begin{lemma}\label{LEM:lifting_bound}
 Assuming that $\tau_\level h_\level \lesssim 1$, we have for all $\lambda \in \skeletalSpace_\level$
 \begin{equation}
  | \liftingOp_\level \lambda |_1 \lesssim \| \localQ_\level \lambda \|_0.
 \end{equation}
\end{lemma}
\begin{proof}
 First, the definition of $\extensionOp_\level$ and Lemma \ref{LEM:lem_35} imply that
 \begin{equation}
  \| \extensionOp_\level \lambda - \localU_\level \lambda \|_0 \lesssim \| \extensionOp_\level \lambda - \localU_\level \lambda \|_\level = \| \lambda - \localU_\level \lambda \|_\level \lesssim h_\level \| \localQ_\level \lambda \|_0.
 \end{equation}
 Additionally, using the inverse inequality, \cite[Lem.~3.3]{ChenLX2014} (stating that $\| \localQ_\level \lambda + \nabla \localU_\level \lambda \|_0 \lesssim h_\level^{-1} \| \localU_\level \lambda - \lambda \|_\level$), and the aforementioned inequality
 \begin{align}\label{EQ:extension_bound}
  | \extensionOp_\level \lambda |_1 \lesssim & h_\level^{-1} | \localU_\level \lambda |_{1,\mesh_\level} + h^{-1}_\level \| \extensionOp_\level \lambda - \localU_\level \lambda \|_0\\
  \lesssim & \| \localQ_\level \lambda \|_0 + h_\level^{-1} \| \localU_\level \lambda - \lambda \|_\level \lesssim \| \localQ_\level \lambda \|_0.\notag
 \end{align}
 Here, $| \cdot |_{1,\mesh_\level}$ denotes the broken (i.e. elementwise) $H^1$-seminorm of an elementwise $H^1$ function. Thus,
 \begin{align}\label{EQ:lifting_bound}
  \| \liftingOp_\level \lambda - \contLinProj_\level \extensionOp_\level \lambda \|_0 \overset{\text{Lem.~\ref{LEM:lifting_identity}}}{~=~} & \| \liftingOp_\level \lambda - \liftingOp_\level \gamma_\level \contLinProj_\level \extensionOp_\level \lambda \|_0 \overset{\eqref{EQ:norm_equiv}}\lesssim \| \lambda - \gamma_\level \contLinProj_\level \extensionOp_\level \lambda \|_\level \notag\\
  ~=~ & \| \gamma_\level \extensionOp_\level \lambda - \gamma_\level \contLinProj_\level \extensionOp_\level \lambda \|_\level \lesssim \| \extensionOp_\level \lambda - \contLinProj_\level \extensionOp_\level \lambda \|_0\notag\\
  \overset{L^2 \text{approx}} \lesssim & h_\level | \extensionOp_\level \lambda |_1,
 \end{align}
 where the second inequality holds since, $\langle .,.\rangle_\level$ is commensurate with the $L^2$ inner product in the bulk domain.
 \begin{align}
  | \liftingOp_\level \lambda |_1 ~\le~ & | \liftingOp_\level \lambda - \contLinProj_\level \extensionOp_\level \lambda |_1 + | \contLinProj_\level \extensionOp_\level \lambda |_1\\
  \overset{\text{inv. ineq.}}{\underset{\text{$H^1$ stab.}}{\lesssim}} & h_\level^{-1} \| \liftingOp_\level \lambda - \contLinProj_\level \extensionOp_\level \lambda \|_0 + | \extensionOp_\level \lambda |_1
  \overset{\eqref{EQ:lifting_bound}} \lesssim~& | \extensionOp_\level \lambda |_1 \overset{\eqref{EQ:extension_bound}} \lesssim \| \localQ_\level \lambda \|_0.\notag
 \end{align}
\end{proof}
\begin{lemma}\label{LEM:tilde_lambda}
 When $\tau_\level = \tfrac{c}{h_\level}$, let for $\lambda \in \skeletalSpace_\level$ $\tilde \lambda_\level \in \skeletalSpace_\level$ be the solution of \eqref{EQ:operator_Poisson}. We have
 \begin{align}
  \| \lambda - \tilde \lambda_\level \|_{a_\level} ~\lesssim~ &  h_\level \| A_\level \lambda \|_\level,\\
  \| \lambda - \tilde \lambda_\level \|_\level ~\lesssim~ & h_\level^2 \| A_\level \lambda \|_\level.
 \end{align}
\end{lemma}

To prove this result, we need the following EDG convergence result
\begin{lemma}\label{LEM:EDG_conv}
 When $\tau_\level = \tfrac{c}{h_\level}$, let for $\tilde \lambda_\level \in \skeletalSpace_\level$ be the solution of \eqref{EQ:operator_Poisson}. We have
 \begin{equation}
  \| \tilde \lambda_\level - \skeletalProj_\level \tilde u\|_\level \lesssim h^2_\level |\tilde u|_2
 \end{equation}
\end{lemma}
\begin{proof}
 Lemma 3.1 in \cite{CockburnGSS2009} states that
 \begin{equation*}
  \| \tilde{\vec q} - \tilde{\vec q}_\level \|^2_0 \lesssim \underbrace{ \| \tilde{\vec q} - \Pi_\level^\textup{RT} \tilde{\vec q} \|^2_0 }_{ \lesssim h^2_\level |\tilde{\vec q}|^2_1 } + \underbrace{ | I_\level^\textup{int} \tilde u - \tilde u |^2_1 }_{ \lesssim h^2_\level |\tilde u|^2_2 } + \frac{1}{\tau_\level} \underbrace{ \nnorm \Pi_\level^\textup{RT} \tilde{\vec q} - \skeletalProj_\level \tilde{\vec q} \nnorm^2_\level }_{ \lesssim h_\level | \tilde{\vec q} |^2_1 },
 \end{equation*}
 where $\tilde{\vec q} = - \nabla \tilde u$, $\Pi_\level^\textup{RT}$ is similar to the standard Raviart--Thomas projection, but has fewer constraints, cf.~\cite[(3.2)]{CockburnGSS2009}. That is, for all $\elem \in \mesh_\level$, the projection $\Pi_\level^\textup{RT}$ suffices
 \begin{align}
  (\Pi_\level^\textup{RT} \tilde{\vec q}, \vec v)_\elem =& (\tilde{\vec q}, \vec v)_\elem && \forall \vec v \in \polynomials^d_{p-1}(\elem), \\
  \langle \Pi_\level^\textup{RT} \tilde{\vec q} \cdot \Nu, \eta \rangle_\face =& \langle \tilde{\vec q} \cdot \Nu, \eta \rangle_\face && \forall \eta \in \polynomials_p(\face),
 \end{align}
 for all $\face \subset \partial \elem$, but one. $I_\level^\textup{int}$ is the continuous interpolant obeying the Dirichlet constraints. Plugging this into \cite[Theo.~2.3]{CockburnGSS2009}, we obtain
 \begin{equation*}
  \| \tilde \lambda_\level - \skeletalProj_\level \tilde u\|_\level \lesssim h_\level \| \tilde{\vec q} - \tilde{\vec q}_\level \|_0 \lesssim h_\level^2 |\tilde u|_2.
 \end{equation*}
\end{proof}
\begin{proof}[Proof of Lemma \ref{LEM:tilde_lambda}]
 Since
 \begin{equation}
  a_\level(\lambda, \mu) = (f_\lambda, \liftingOp_\level \mu)_0 \quad \text{and} \quad a_\level (\tilde \lambda_\level, \mu ) = (f_\lambda, \localU_\level \mu)_0,
 \end{equation}
 we have
 \begin{equation}\label{EQ:tildelambdadifference}
  a_\level (\lambda - \tilde \lambda_\level, \mu) = (f_\lambda, \liftingOp_\level \mu - \localU_\level \mu)_0.
 \end{equation}
 This yields for all $\mu \in \skeletalSpace_\level$
 \begin{align}\label{EQ:liftung_u_diff}
  \| \liftingOp_\level \mu - \localU_\level \mu \|_0 ~=~ & \| \liftingOp_\level \mu - \discProj_\level \liftingOp_\level \mu \|_0\\
  \overset{L^2 \text{ approx.}} \lesssim~ & h_\level | \liftingOp_\level \mu |_1 \overset{\text{Lem.~}\ref{LEM:lifting_bound}} \lesssim h_\level \| \localQ_\level \mu \|_0 \le h_\level \| \mu \|_{a_\level}.\notag
 \end{align}
 Setting $\mu = \lambda - \tilde \lambda_\level$ in \eqref{EQ:tildelambdadifference}, we get
 \begin{equation}
  \| \lambda - \tilde \lambda_\level \|^2_{a_\level} \lesssim \| f_\lambda \|_0 \; h_\level \| \lambda - \tilde \lambda_\level \|_{a_\level}.
 \end{equation}
 This also implies that
 \begin{equation}
  \| \lambda - \tilde \lambda_\level \|_{a_\level} \lesssim h_\level \| f_\lambda \|_0 \overset{\eqref{EQ:f_lambda_bound}} \lesssim h_\level \| A_\level \lambda \|_\level.
 \end{equation}
 This is the first inequality.
 
 In the following, we will utilize the duality argument to prove the lemma's second inequality: Suppose
 \begin{equation}
  - \Delta \psi = \liftingOp_\level (\lambda - \tilde \lambda_\level) \text{ in } \Omega \quad \text{and} \quad \psi = 0 \text{ on } \partial \Omega
 \end{equation}
 and that $\tilde \rho_\level$ is the EDG approximation of $\psi$ on the skeleton, which means
 \begin{equation}
  a_\level (\tilde \rho_\level, \mu) = (\liftingOp_\level (\lambda - \tilde \lambda_\level), \localU_\level \mu) \qquad \forall \mu \in \skeletalSpace_\level.
 \end{equation}
 Moreover, let $\rho_\level$ be the solution of
 \begin{equation}\label{EQ:define_rho}
  a_\level (\rho_\level, \mu) = (\liftingOp_\level (\lambda - \tilde \lambda_\level), \liftingOp_\level \mu) \qquad \forall \mu \in \skeletalSpace_\level.
 \end{equation}
 Similar to the estimation of $\| \lambda - \tilde \lambda_\level \|_{a_\level}$, we have
 \begin{equation}\label{EQ:rho_difference}
  \| \rho_\level - \tilde \rho_\level \|_{a_\level} \lesssim h_\level \| \liftingOp_\level (\lambda - \tilde \lambda_\level) \|_0.
 \end{equation}
 Taking $\mu = \lambda - \tilde \lambda_\level$ in \eqref{EQ:define_rho}, we receive
 \begin{align}\label{EQ:lifting_difference}
  \| \liftingOp_\level (\lambda - \tilde \lambda_\level) \|^2_0 ~=~ & a_\level (\rho_\level, \lambda - \tilde \lambda_\level) \overset{\eqref{EQ:tildelambdadifference}}{\underset{\text{Lem.~\ref{LEM:lifting_identity}}}{~=~}} a_\level (\lambda - \tilde \lambda_\level, \rho_\level - \gamma_\level\contLinProj_\level \psi)\\
  ~=~ & a_\level (\rho_\level - \tilde \rho_\level , \lambda - \tilde \lambda_\level) + a_\level (\tilde \rho_\level - \gamma_\level \contLinProj_\level \psi, \lambda - \tilde \lambda_\level)\notag\\
  ~\le~ & \| \lambda - \tilde \lambda_\level \|_{a_\level} \left( \| \rho_\level - \tilde \rho_\level \|_{a_\level} + \| \tilde \rho_\level - \gamma_\level \contLinProj_\level \psi \|_{a_\level} \right).\notag
 \end{align}
 This can be further estimated noting that
 \begin{align}
  \| \tilde \rho_\level - \gamma_\level \contLinProj_\level \psi \|_{a_\level} \overset{\text{Lem.~\ref{LEM:eigenvalue_bound}}} \lesssim & h_\level^{-1} \| \tilde \rho_\level - \gamma_\level \contLinProj_\level \psi \|_\level \\
  \lesssim~ & h_\level^{-1} \left( \| \tilde \rho_\level - \skeletalProj_\level \psi \|_\level + \| \skeletalProj_\level \psi - \psi \|_\level + \| \psi - \contLinProj_\level \psi \|_\level \right)\notag\\
  \overset{\text{Lem.~\ref{LEM:EDG_conv}}}{\underset{\text{trace \& $L^2$ approx}}{~\lesssim~}} & h_\level^{-1} \left( h_\level^2 | \psi |_2 \right) \overset{\text{regularity}}{\lesssim} h_\level \| \liftingOp_\level (\lambda - \tilde \lambda_\level) \|_0.\notag
 \end{align}
 Using this inequality combined with \eqref{EQ:rho_difference} and \eqref{EQ:lifting_difference}, we have
 \begin{equation}
  \| \liftingOp_\level (\lambda - \tilde \lambda_\level) \|^2_0 \lesssim h_\level \| \liftingOp_\level (\lambda - \tilde \lambda_\level) \|_0 \| \lambda - \tilde \lambda_\level \|_{a_\level}.
 \end{equation}
 By the lemma's first inequality and \eqref{EQ:norm_equiv}
 \begin{equation}
  \| \lambda - \tilde \lambda_\level \|_\level \cong \| \liftingOp_\level (\lambda - \tilde \lambda_\level) \|_0 \lesssim h_\level \| \lambda - \tilde \lambda_\level \|_{a_\level} \lesssim h_\level^2 \| A_\level \lambda \|_\level.
 \end{equation}
\end{proof}
\begin{lemma}\label{LEM:injection_properties}
 Provided $\tau_\level h_\level \lesssim 1$, the injection operator $\injectionOp_\level \colon \skeletalSpace_{\level - 1} \to \skeletalSpace_\level$ satisfies
 \begin{align}
  \injectionOp_\level \gamma_{\level - 1} w ~=~ & \gamma_\level w && \text{if } w \in \linElementSpace_{\level - 1},\\
  \| \injectionOp_\level \mu \|_\level ~\lesssim~ & \| \mu \|_{\level - 1} && \forall \mu \in \skeletalSpace_{\level - 1}.
 \end{align}
\end{lemma}
\begin{proof}
 One easily verifies that for all $w \in \linElementSpace_{\level - 1}$
 \begin{equation}
  \extensionOp_{\level - 1} \gamma_{\level - 1} w = w.
 \end{equation}
 This implies the equality. By the standard scaling argument and Lemma \ref{LEM:thm_31} we can also get the inequality via
 \begin{equation}
  \| \injectionOp_\level \mu \|_\level = \| \gamma_\level \extensionOp_{\level-1} \mu \|_\level \lesssim \| \extensionOp_{\level - 1} \mu \|_0 \lesssim \| \mu \|_{\level - 1}.
 \end{equation}
\end{proof}

\begin{lemma}\label{LEM:square_approx}
 Suppose that $w \in H^2(\Omega)$ is the solution of
 \begin{equation}
  - \Delta w = g \text{ in } \Omega, \qquad w = 0 \text{ on } \partial \Omega,
 \end{equation}
 where $g \in L^2(\Omega)$, and $w_\level$, $w_{\level -1}$ are its EDG approximates with respect to $\mesh_\level$ and $\mesh_{\level -1}$, respectively. If $\tau_\level = \tfrac{c}{h_\level}$, we obtain
 \begin{equation}
  \| w_\level - \injectionOp_\level w_{\level -1} \|_\level \lesssim h^2_\level \| g \|_0.
 \end{equation}
\end{lemma}
\begin{proof}
 Defining $\hat w_{\level -1} \in \linElementSpace_{\level - 1}$ as the CG solution of $w$ and $\gamma_{\level - 1} \hat w_{\level - 1}$ as its trace on $\skeletal_{\level -1}$, we obtain
 \begin{gather}
  \| w_\level - \injectionOp_\level w_{\level -1} \|_\level \\
  \le \| w_\level - \skeletalProj_\level w \|_\level + \| \skeletalProj_\level w - \injectionOp_\level \gamma_{\level - 1} \hat w_{\level -1} \|_\level + \| \injectionOp_\level \gamma_{\level - 1} \hat w_{\level -1} - \injectionOp_\level w_{\level -1} \|_\level. \notag
 \end{gather}
 Then, by the approximation properties of EDG (cf. Lemma \ref{LEM:EDG_conv}), CG, inverse inequality, $\| v \|_\level \lesssim \| v \|_0$ for $v \in \contElementSpace_\level$, and Lemma~\ref{LEM:injection_properties}, we obtain
 \begin{align}
  \| w_\level - \skeletalProj_\level w \|_\level &~\lesssim~ h^2_\level \| g \|_0,\\
  \| \skeletalProj_\level w - \injectionOp_\level \gamma_{\level - 1} \hat w_{\level -1} \|_\level & ~=~ \| \skeletalProj_\level w - \gamma_\level \hat w_{\level -1} \|_\level \lesssim h^2_\level \| g \|_0,\\
  \| \injectionOp_\level \gamma_{\level - 1} \hat w_{\level -1} - \injectionOp_\level w_{\level -1} \|_\level & ~\lesssim~\| \gamma_{\level - 1} \hat w_{\level -1} - w_{\level -1} \|_{\level - 1} \\
  &~\le~\| \gamma_{\level - 1} \hat w_{\level -1} - \skeletalProj_{\level - 1} w \|_\level + \| \skeletalProj_{\level - 1} w - w_{\level -1} \|_{\level - 1} \notag\\
  & ~\lesssim~ h^2_\level \| g \|_0.\notag
 \end{align}
 This gives the result.
\end{proof}
\begin{lemma}
 \eqref{EQ:precond14} holds if $\tau_\level \cong h_\level^{-1}$.
\end{lemma}
\begin{proof}
 Let $\tilde \lambda_{\level - 1} \in \skeletalSpace_{\level - 1}$ be the solution of
 \begin{equation}
  a_{\level - 1} (\tilde \lambda_{\level - 1}, \mu) = (f_\lambda, \localU_{\level - 1} \mu) \qquad \forall \mu \in \skeletalSpace_{\level - 1},
 \end{equation}
 i.e., $\tilde \lambda_{\level - 1}$ is the EDG solution of \eqref{EQ:Poisson} on $\skeletalSpace_{\level - 1}$. By Lemma \ref{LEM:square_approx} and \eqref{EQ:f_lambda_bound}, we have
 \begin{equation}\label{EQ:tilde_injection}
  \| \tilde \lambda_\level - \injectionOp_\level \tilde \lambda_{\level - 1} \|_\level \lesssim h_\level^2 \| f_\lambda \|_0 \lesssim h_\level^2 \| A_\level \lambda \|_\level.
 \end{equation}
 Denoting $e_{\level - 1} = \tilde \lambda_{\level - 1} - \projectionOp_{\level - 1} \lambda$, we can conclude via 
 \begin{equation}\label{EQ:weak_error}
  a_{\level - 1} (\projectionOp_{\level - 1} \lambda, \mu) = a_\level (\lambda, \injectionOp_\level \mu) = (f_\lambda, \liftingOp_\level \injectionOp_\level \mu ) \qquad \forall \mu \in \skeletalSpace_{\level - 1}
 \end{equation}
 that
 \begin{equation}\label{EQ:error_f_lambda}
  a_{\level - 1}(e_{\level - 1}, \mu) = (f_\lambda , \localU_{\level - 1} \mu - \liftingOp_\level \injectionOp_\level \mu)_0.
 \end{equation}
 Noting that for $\mu = \gamma_{\level - 1} w$ with $w \in \linElementSpace_{\level - 1}$ by Lemmas \ref{LEM:lifting_identity} \& \ref{LEM:injection_properties}
 \begin{equation}
  \localU_{\level - 1} \mu = w = \liftingOp_\level \injectionOp_\level \mu,
 \end{equation}
 which means $a_{\level - 1} (e_{\level - 1} , \mu ) = 0$. Similar to the estimation in Lemma \ref{LEM:tilde_lambda} we can use the duality argument to receive
 \begin{equation}\label{EQ:norm_bound}
  \| e_{\level - 1} \|_{\level - 1} \lesssim h_{\level - 1} \| e_{\level - 1} \|_{a_{\level - 1}}.
 \end{equation}
 Thus, we can deduce that for all $\mu \in \skeletalSpace_{\level - 1}$
 \begin{align}\label{EQ:error_approx}
  \| \liftingOp_\level \injectionOp_\level \mu - \localU_{\level - 1} \mu \|_0 ~\le~ & \| \liftingOp_\level \injectionOp_\level \mu - \localU_\level \injectionOp_\level \mu \|_0 + \| \localU_\level \injectionOp_\level \mu - \localU_\level \mu \|_0 \\
  \lesssim~& h_\level \| \injectionOp_\level \mu \|_{a_\level} + h_\level \| \mu \|_{a_{\level - 1}} \lesssim h_\level \| \mu \|_{a_{\level - 1}}.\notag
 \end{align}
 Here, the second inequality is obtained using \eqref{EQ:liftung_u_diff} and Lemma \ref{LEM:lem_43} and the last inequality is Lemma~\ref{LEM:injection_bounded}.
 
 Taking $\mu = e_{\level - 1}$ in \eqref{EQ:error_f_lambda} and using \eqref{EQ:error_approx}, we have
 \begin{equation}
  \| e_{\level - 1} \|^2_{a_{\level - 1}} \lesssim \| f_\lambda \|_0 \; h_\level \| e_{\level - 1} \|_{a_{\level - 1}},
 \end{equation}
 that is
 \begin{equation}
  \| e_{\level - 1} \|_{a_{\level - 1}} \lesssim h_\level \| f_\lambda \|_0 \lesssim h_\level \| A_\level \lambda \|_\level.
 \end{equation}
 Using \eqref{EQ:norm_bound}, this results in
 \begin{equation}\label{EQ:error_square}
  \| e_{\level - 1} \|_{\level - 1} \lesssim h_\level^2 \|A_\level \lambda \|_{\level}
 \end{equation}
 and by triangle inequality (first inequality), Lemma \ref{LEM:tilde_lambda} \& \ref{LEM:injection_properties}, \eqref{EQ:tilde_injection} (second inequality), and \eqref{EQ:error_square} (last inequality)
 \begin{align}
  \| \lambda - \injectionOp_\level \projectionOp_{\level - 1} \lambda \|_\level ~\le~& \| \lambda - \tilde \lambda_\level \|_\level + \| \tilde \lambda_\level - \injectionOp_\level \tilde \lambda_{\level - 1} \|_\level + \| \injectionOp_\level \tilde \lambda_{\level - 1} - \injectionOp_\level \projectionOp_{\level - 1} \lambda \|_\level \notag\\
  \lesssim~ & h_\level^2 \| A_\level \lambda \|_\level + \| e_{\level - 1} \|_{\level - 1} \lesssim h_\level^2 \| A_\level \lambda \|_\level.
 \end{align}
\end{proof}
%
\section{Proof of \eqref{EQ:precond2} and \eqref{EQ:precond3}}\label{SEC:proof_a2_a3}
%
The proof of \eqref{EQ:precond2} is a simple consequence of Lemma \ref{LEM:injection_bounded} with $\projectionOp_{\level-1} \lambda$ instead of $\lambda$ and the following lemma which can be obtained similar to \cite[Lem.~4.2]{LuRK2020}:
\begin{lemma}\label{LEM:projection_stable}
 The ``Ritz quasi-projection'' $\projectionOp_{\level-1}: \skeletalSpace_\level \to  \skeletalSpace_{\level-1}$ is stable in the sense that for all $\lambda \in \skeletalSpace_\level$, we have
 \begin{equation}
  \| \projectionOp_{\level-1}  \lambda \|_{a_{\level-1}} \lesssim \| \lambda \|_{a_\level}.
 \end{equation}
\end{lemma}
Thus, we can deduce that
\begin{align}
 a_\level ( \lambda - & \injectionOp_{\level-1} \projectionOp_\level \lambda, \lambda - \injectionOp_\level \projectionOp_{\level-1} \lambda ) \\
 = & a_\level ( \lambda, \lambda ) - 2 a_\level ( \lambda, \injectionOp_\level \projectionOp_{\level-1} \lambda ) + a_\level ( \injectionOp_\level \projectionOp_{\level-1} \lambda, \injectionOp_\level \projectionOp_{\level-1} \lambda ) \notag\\
 \le & a_\level ( \lambda, \lambda ) \underbrace{ - 2 a_{\level - 1} ( \projectionOp_{\level-1} \lambda,  \projectionOp_{\level-1} \lambda ) }_{ \le 0 } + C \underbrace{ a_{\level-1} ( \projectionOp_{\level-1} \lambda, \projectionOp_{\level-1} \lambda ) }_{ \lesssim \| \lambda \|^2_{a_\level} \; \text{by Lemma \ref{LEM:projection_stable}} },\notag
\end{align}

For the proof of \eqref{EQ:precond3}, we heavily rely on \cite{BrambleP1992} (where \eqref{EQ:precond3} is denoted (2.11)). Theorems 3.1 and 3.2 of \cite{BrambleP1992} ensure that \eqref{EQ:precond3} holds if the subspaces satisfy a ``limited interaction property'' which holds, because each degree of freedom (DoF) only ``communicates'' with other DoFs which are located on the same face as the DoF or on the other faces of the two adjacent elements.
%
\section{Numerical experiments}\label{SEC:numerics}
%
\begin{figure}[t!]
 \begin{tikzpicture}[scale = 2.]
  \draw (0,2) -- (0,0) -- (1,0) -- (1,1) -- (0,1) -- (1,0) -- (2,0) -- (2,2) -- (1,2) -- (1,1) -- (2,1) -- (1,2) -- (0,2) -- (2,0);
 \end{tikzpicture}
 \caption{Coarse grid (level 0) for numerical experiments. Meshes on higher levels are generated by uniform refinement.}\label{FIG:mesh}
\end{figure}
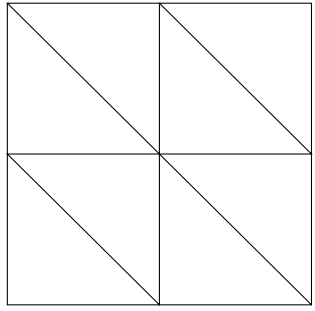
To evaluate the multigrid method for EDG numerically, we consider the Poisson problem
\begin{subequations}\label{EQ:num_testcase}
\begin{align}
 -\Delta u & = f && \text{ in } \Omega,\\
 u & = 0 && \text{ on } \partial \Omega,
\end{align}
\end{subequations}
%
where $f$ is chosen as one on the unit square $\Omega = (0,1)^2$. The implementation is based on the FFW toolbox from \cite{BGGRW07} and employs the Gauss--Seidel smoother. It uses a Lagrange basis and the Euclidean inner product in the coefficient space instead of the inner product $\langle .,. \rangle_\level$. These two inner products are equivalent up to a factor of $h^2_\level$. The numerical experiments are conducted on a successively refined mesh sequence of which the initial mesh is depicted in Figure \ref{FIG:mesh}. The iteration process for approximating $\vec x$ in $\mathbf A \vec x = \vec b$ representing the discrete version of \eqref{EQ:num_testcase} is stopped if
\begin{equation}
 \frac{ \| \vec b - \mathbf A \vec x_\text{iter} \|_2 }{ \| \vec b \|_2 } < 10^{-6},
\end{equation}
and the initial value $\vec x$ on mesh level $\ell$ is the solution on level $\ell-1$, describing a nested iteration.
The numbers of iteration steps are shown in
Table~\ref{TAB:multigrid_steps_f1} 
and appear to be independent of the mesh level, as predicted by our
analysis. Additionally, the numbers are fairly small, such that we can
conclude that we actually have an efficient method. 
Finally, we see that the choice of
$\tau \in \{ \tfrac{1}{h}, 1
\}$ does not significantly influence the number of iterations. We did
experiments for polynomial degrees up to three, and observed that the
iteration counts remain well bounded; nevertheless, we expect rising
counts for higher degrees, as we use a point smoother.
\begin{table}
 \begin{tabular}{cc|cccccc|cccccc}
  \toprule
  \multicolumn{2}{c|}{smoother}    & \multicolumn{6}{c|}{one step} & \multicolumn{6}{c}{two steps} \\
  \midrule
  \multicolumn{2}{c|}{mesh level}  & 1 & 2 & 3 & 4 & 5 & 6 & 1 & 2 & 3 & 4 & 5 & 6 \\
  \midrule
  \multirow{2}{*}{\rotatebox[origin=c]{90}{$p = 1$}}
  & $\tau = \tfrac{1}{h}$ & 6 & 7 & 7 & 6 & 6 & 6 & 4 & 5 & 5 & 5 & 4 & 4 \\
  & $\tau = 1$            & 6 & 7 & 7 & 6 & 6 & 6 & 4 & 5 & 5 & 5 & 4 & 4 \\
  \midrule
  \multirow{2}{*}{\rotatebox[origin=c]{90}{$p = 2$}}
  & $\tau = \tfrac{1}{h}$ & 7 & 7 & 7 & 7 & 7 & 7 & 5 & 4 & 4 & 4 & 4 & 4 \\
  & $\tau = 1$            & 7 & 7 & 7 & 7 & 7 & 7 & 5 & 4 & 4 & 4 & 4 & 4 \\
  \midrule
  \multirow{2}{*}{\rotatebox[origin=c]{90}{$p = 3$}}
  & $\tau = \tfrac{1}{h}$ & 9 & 9 & 9 & 9 & 9 & 9 & 6 & 6 & 6 & 6 & 5 & 5 \\
  & $\tau = 1$            & 9 & 9 & 9 & 9 & 9 & 9 & 6 & 6 & 6 & 6 & 5 & 5 \\
  \bottomrule
 \end{tabular}\vspace{1ex}
 \caption{Numbers of iterations with one and two smoothing steps for $f \equiv 1$. The polynomial degree of the EDG method is $p$.}\label{TAB:multigrid_steps_f1}
\end{table}

Additionally, we tested the correctness of our implementation by
employing a right hand side leading to the solution
$u = \sin(2 \pi x) \sin(2 \pi y)$. The estimated orders of convergence
(EOC) of the primary unknown $u$ computed as
\begin{equation}
 \text{EOC} = \log \left( \frac{ \| u - u_{\level-1} \|_{L^2(\Omega)} }{ \| u - u_\level \|_{L^2(\Omega)} } \right) / \log (2),
\end{equation}
and the secondary unknown $\vec q$ of the HDG method are reported in
Table~\ref{TAB:hdg_eoc}. 
Iteration counts are almost
identical to those in Table~\ref{TAB:multigrid_steps_f1}, such that we
do not report them here.  As opposed to the HDG method, we see that the choice $\tau = \tfrac{1}{h}$
is not suboptimal as compared to $\tau = 1$. 
\begin{table}
 \begin{tabular}{cc|@{\,}lcc@{\,}lcc@{\,}lcc@{\,}lcc@{\,}lcc@{\,}lcc}
  \toprule
  \multicolumn{2}{c|@{\,}}{mesh}                  && \multicolumn{2}{c}{2}  && \multicolumn{2}{c}{3}   && \multicolumn{2}{c}{4}   && \multicolumn{2}{c}{5}    && \multicolumn{2}{c}{6}     && \multicolumn{2}{c}{7}     \\
  \cmidrule{4-5} \cmidrule{7-8} \cmidrule{10-11} \cmidrule{13-14} \cmidrule{16-17} \cmidrule{19-20}
  \multicolumn{2}{c|@{\,}}{EOC}                   && $u$  & $\vec q$  && $u$  & $\vec q$  && $u$  & $\vec q$  && $u$  & $\vec q$  && $u$  & $\vec q$  && $u$  & $\vec q$  \\
  \midrule
  \multirow{2}{*}{\rotatebox[origin=c]{90}{$p = 1$}}
  & $\tau = \tfrac{1}{h}$ && 0.8 & 0.5 && 1.6 & 0.8 && 1.9 & 1.0 && 2.0 & 1.0 && 2.0 & 1.0 && 2.0 & 1.0 \\
  & $\tau = 1$            && 0.8 & 0.5 && 1.6 & 0.8 && 1.9 & 1.0 && 2.0 & 1.0 && 2.0 & 1.0 && 2.0 & 1.0 \\
  \midrule
  \multirow{2}{*}{\rotatebox[origin=c]{90}{$p = 2$}}
  & $\tau = \tfrac{1}{h}$ && 3.0 & 1.5 && 3.0 & 1.8 && 3.0 & 1.9 && 3.0 & 2.0 && 3.0 & 2.0 && 3.0 & 2.0 \\
  & $\tau = 1$            && 3.0 & 1.5 && 3.0 & 1.8 && 3.0 & 1.9 && 3.0 & 2.0 && 3.0 & 2.0 && 3.0 & 2.0 \\
  \midrule
  \multirow{2}{*}{\rotatebox[origin=c]{90}{$p = 3$}}
  & $\tau = \tfrac{1}{h}$ && 4.0 & 2.8 && 4.2 & 2.9 && 4.2 & 3.0 && 4.0 & 3.0 && 4.0 & 3.0 && 4.0 & 3.0 \\
  & $\tau = 1$            && 3.1 & 2.8 && 3.9 & 2.9 && 4.0 & 3.0 && 4.0 & 3.0 && 4.0 & 3.0 && 4.0 & 3.0 \\
  \bottomrule
 \end{tabular}\vspace{1ex}
 \caption{Estimated orders of convergence (EOC) for primary unknown $u$ and secondary unknown $\vec q$ when the polynomial degree of the EDG method is $p$ and $u = \sin(2 \pi x) \sin(2 \pi y)$.}\label{TAB:hdg_eoc}
\end{table}
%
\section{Conclusions}
%
In the previous pages, we proposed a homogeneous multigrid method for EDG. We proved analytically that this method converges independently of the mesh size. Numerical examples have shown that the condition numbers are not only independent of the mesh size but also reasonably small. As as consequence, we have been enabled to efficiently solve linear systems of equations arising from EDG discretizations of arbitrary order.
%
\appendix\section{Used results}
%
Here, we summarize the results from other sources that we used in the proofs of our propositions.
\begin{lemma}\label{LEM:lem_35}
 Let $\mu$ be any function in $\skeletalSpace_\level$. The following statement holds:
 \begin{equation}
  \nnorm \sqrt{\tau_\level} (\localU_\level \mu - \mu) \nnorm_\level \lesssim \sqrt{ h_\level \tau_\level } \| \localQ_\level \mu \|_0.
 \end{equation}
 Thus, if $\tau_\level h_\level \lesssim 1$,
 \begin{equation}
  \| \localU_\level \mu - \mu \|_\level \lesssim h_\level \| \localQ_\level \mu \|_0.
 \end{equation}
\end{lemma}
\begin{proof}
 The first inequality is \cite[Lemma 3.4 (iv)]{CockburnDGT2013} whose right hand side is estimated using \cite[Lemma 3.4 (v)]{CockburnDGT2013}. The second inequality follows after multiplication with $h_\level$ and exploiting the definitions of $\nnorm \cdot \nnorm_\level$ and $\| \cdot \|_\level$.
\end{proof}
\begin{lemma}\label{LEM:thm_31}
 If $\tau_\level h_\level \lesssim 1$, the local solution operators obeys
 \begin{equation}
  \| \localU_\level \mu \|_0 \lesssim \| \mu \|_\level
 \end{equation}
\end{lemma}
\begin{proof}
 This is Theorem 3.1 in \cite{CockburnDGT2013}, where we use that the constant becomes independent of $h_\level$ if $\tau_\level h_\level \lesssim 1$.
\end{proof}
%
%
\begin{lemma}[Lemma 4.3 in \cite{LuRK2020}]
  \label{LEM:lem_43}
 The DG reconstructions of the injection operator admits the estimate
 \begin{equation}
   \| \localU_{\level-1} \mu - \localU_\level \injectionOp_\level \mu \|_0 \lesssim h_\level \| \mu \|_{a_{\level-1}},
   \qquad
   \forall \mu \in \skeletalSpace_{\level-1}.
 \end{equation}
\end{lemma}

\bibliographystyle{alpha}
\bibliography{MultigridEDG}

\newcommand{\etalchar}[1]{$^{#1}$}
\begin{thebibliography}{CDGT13}

\bibitem[BGG{\etalchar{+}}]{BGGRW07}
A.~Byfut, J.~Gedicke, D.~Günther, J.~Reininghaus, and S.~Wiedemann.
\newblock {FFW} documentation.
\newblock https://github.com/project-openffw/openffw.

\bibitem[BP92]{BrambleP1992}
J.H. Bramble and J.E. Pasciak.
\newblock The analysis of smoothers for multigrid algorithms.
\newblock {\em Mathematics of Computation}, 58(198):467--488, 1992.

\bibitem[BPX91]{BramblePX1991}
J.H. Bramble, J.E. Pasciak, and J.~Xu.
\newblock The analysis of multigrid algorithms with nonnested spaces or
  noninherited quadratic forms.
\newblock {\em Mathematics of Computation}, 56(193):1--34, 1991.

\bibitem[CDGT13]{CockburnDGT2013}
B.~Cockburn, O.~Dubois, J.~Gopalakrishnan, and S.~Tan.
\newblock Multigrid for an {HDG} method.
\newblock {\em IMA Journal of Numerical Analysis}, 34(4):1386--1425, 10 2013.

\bibitem[CFSZ19]{Chen2019}
G.~Chen, G.~Fu, J.~R. Singler, and Y.~Zhang.
\newblock A class of embedded {DG} methods for {D}irichlet boundary control of
  convection diffusion {PDE}s.
\newblock {\em Journal of Scientific Computing}, pages 1--26, 2019.

\bibitem[CGL09]{CockburnGL2009}
B.~Cockburn, J.~Gopalakrishnan, and R.~Lazarov.
\newblock Unified hybridization of discontinuous {G}alerkin, mixed, and
  continuous {G}alerkin methods for second order elliptic problems.
\newblock {\em SIAM Journal on Numerical Analysis}, 47(2):1319--1365, 2009.

\bibitem[CGSS09]{CockburnGSS2009}
B.~Cockburn, J.~Guzmán, S.C. Soon, and H.K. Stolarski.
\newblock An analysis of the embedded discontinuous {G}alerkin method for
  second-order elliptic problems.
\newblock {\em SIAM Journal on Numerical Analysis}, 47(4):2686--2707, 2009.

\bibitem[CLX14]{ChenLX2014}
H.~Chen, P.~Lu, and X.~Xu.
\newblock A robust multilevel method for hybridizable discontinuous {G}alerkin
  method for the {H}elmholtz equation.
\newblock {\em Journal of Computational Physics}, 264:133--151, 2014.

\bibitem[DGTZ07]{DuanGTZ2007}
H.Y. Duan, S.Q. Gao, R.C.E. Tan, and S.~Zhang.
\newblock A generalized {BPX} multigrid framework covering nonnested {V}-cycle
  methods.
\newblock {\em Mathematics of Computation}, 76(257):137--152, 2007.

\bibitem[FS17]{Fu2017}
G.~Fu and C.W. Shu.
\newblock Analysis of an embedded discontinuous {G}alerkin method with
  implicit-explicit time-marching for convection-diffusion problems.
\newblock {\em International Journal of Numerical Analysis \& Modeling},
  14(4):477--499, 2017.

\bibitem[GR86]{GiraultR1986}
V.~Girault and P.A. Raviart.
\newblock {\em Finite Element Methods for {N}avier-{S}tokes Equations}.
\newblock Springer-Verlag, Berlin Heidelberg, 1986.

\bibitem[GT09]{GopalakrishnanTan09}
J.~Gopalakrishnan and S.~Tan.
\newblock A convergent multigrid cycle for the hybridized mixed method.
\newblock {\em Numerical Linear Algebra with Applications}, 16:689--714, 2009.

\bibitem[Kam16]{Kamenetskiy2016}
D.~S. Kamenetskiy.
\newblock On the relation of the embedded discontinuous {G}alerkin method to
  the stabilized residual-based finite element methods.
\newblock {\em Applied Numerical Mathematics}, 108:271--285, 2016.

\bibitem[LRK20]{LuRK2020}
P.~Lu, A.~Rupp, and G.~Kanschat.
\newblock {HMG} --- {H}omogeneous multigrid for {HDG}.
\newblock {\em arXiv preprint arXiv:2011.14018}, page~16, 2020.

\bibitem[NPC15]{NguyenPC2015}
N.~C. Nguyen, J.~Peraire, and B.~Cockburn.
\newblock A class of embedded discontinuous {G}alerkin methods for
  computational fluid dynamics.
\newblock {\em Journal of Computational Physics}, 302:674--692, 2015.

\bibitem[PNC11]{PeraireNC2011}
J.~Peraire, N.~C. Nguyen, and B.~Cockburn.
\newblock An embedded discontinuous {G}alerkin method for the compressible
  {E}uler and {N}avier-{S}tokes equations.
\newblock In {\em 20th AIAA Computational Fluid Dynamics Conference}, page
  3228, 2011.

\bibitem[RW20]{Rhebergen2020}
S.~Rhebergen and G.~N. Wells.
\newblock An embedded--hybridized discontinuous {G}alerkin finite element
  method for the {S}tokes equations.
\newblock {\em Computer Methods in Applied Mechanics and Engineering},
  358:112619, 2020.

\bibitem[Tan09]{TanPhD}
S.~Tan.
\newblock {\em Iterative solvers for hybridized finite element methods}.
\newblock PhD thesis, University of Florida, 2009.

\bibitem[ZZS18]{Zhang2018}
X.~Zhang, Y.~Zhang, and J.~R. Singler.
\newblock An {EDG} method for distributed optimal control of elliptic {PDE}s.
\newblock {\em arXiv preprint arXiv:1801.02978}, 2018.

\end{thebibliography}
\end{document}